\documentclass[12pt]{amsart}
\usepackage{amsmath}
\usepackage{amsxtra}
\usepackage{amscd}
\usepackage{amsthm}
\usepackage{amsfonts}
\usepackage{amssymb}
\usepackage {pstricks}
\usepackage{pstricks,pst-node}
\usepackage[all]{xy}

\usepackage{verbatim}

\newtheorem{theorem}{Theorem}[section]

\newtheorem{lemma}[theorem]{Lemma}

\newtheorem{corollary}[theorem]{Corollary}

\newtheorem{fact}[theorem]{Assertion}

\theoremstyle{definition}
\newtheorem{definition}[theorem]{Definition}
\newtheorem{example}[theorem]{Example}

\theoremstyle{remark}
\newtheorem{remark}[theorem]{Remark}

\numberwithin{equation}{section}

\newcommand{\mc}{\mathcal}

\newcommand{\be}{\begin{equation}}
\newcommand{\ee}{\end{equation}}

\newcommand{\C}{{\mathbb C}}

\newcommand{\Z}{{\mathbb Z}}

\newcommand{\X}{{\mathbb X}}

\newcommand{\id}{{\rm{id}}}

\newcommand{\mf}{\mathfrak}
\newcommand{\fg}{{\mf g}}

\newcommand{\fb}{{\mf b}}

\newcommand{\cI}{\mc I}

\newcommand{\cS}{\mc S}

\newcommand{\tL}{\tilde{L}}
\newcommand{\inv}{^{-1}}

\newcommand{\U}{{\rm{U}}}
\newcommand{\End}{{\rm{End}}}
\newcommand{\Hom}{{\rm{Hom}}}

\newcommand{\GL}{{\rm{GL}}}
\newcommand{\Sym}{{\rm{Sym}}}

\newcommand{\tr}{{\rm{tr}}}

\advance\headheight by 2pt

\newcommand{\fso}{{\mathfrak {so}}}

\newcommand{\gl}{{\mathfrak {gl}}}

\newcommand{\osp}{{\mathfrak {osp}}}

\newcommand{\Sp}{{\rm Sp}}

\newcommand{\ot}{\otimes}
\newcommand{\lr}{\longrightarrow}

\newcommand{\sdim}{{\rm sdim\,}}

\newcommand{\OSp}{{\rm OSp}}
\newcommand{\SOSp}{{\rm SOSp}}
\newcommand{\SO}{{\rm SO}}

\begin{document}
\definecolor{bondiblue}{rgb}{0.0, 0.58, 0.71}
\definecolor{brown(traditional)}{rgb}{0.59, 0.29, 0.0}
\normalfont

\title[Invariants of the orthosymplectic Lie superalgebra]{Invariants of the orthosymplectic Lie superalgebra
and super Pfaffians}

\author{G.I. Lehrer and R.B. Zhang}
\thanks{This research was supported by the Australian Research Council}
\address{School of Mathematics and Statistics,
University of Sydney, N.S.W. 2006, Australia}
\email{gustav.lehrer@sydney.edu.au, ruibin.zhang@sydney.edu.au}
\begin{abstract}
{
Given a complex orthosymplectic superspace $V$,
the orthosymplectic Lie superalgebra $\osp(V)$ and general linear algebra $\gl_N$ both act naturally
on the coordinate super-ring $\cS(N)$ of the dual space of $V\otimes\C^N$, and their actions commute.
Hence the subalgebra $\cS(N)^{\osp(V)}$ of $\osp(V)$-invariants in $\cS(N)$
has a $\gl_N$-module structure.
}
We introduce the space of super Pfaffians
as a simple $\gl_N$-submodule of $\cS(N)^{\osp(V)}$,
give an explicit formula for its highest weight vector,
and show that the super Pfaffians and the elementary
(or `Brauer') $\OSp(V)$-invariants
together generate $\cS(N)^{\osp(V)}$ as an algebra.
The decomposition of $\cS(N)^{\osp(V)}$
as a direct sum of simple $\gl_N$-submodules is obtained
and shown to be multiplicity free. Using Howe's $(\gl(V), \gl_N)$-duality on
$\cS(N)$, we deduce from the decomposition that the subspace of $\osp(V)$-invariants in any simple
$\gl(V)$-tensor module is either $0$ or $1$-dimensional.
These results also enable us to determine the
$\osp(V)$-invariants in the tensor powers $V^{\otimes r}$ for all $r$.
\end{abstract}
\maketitle

\tableofcontents

\section{Introduction}\label{sect:intro}

In \cite{LZ5, LZ6}, we established the first and second fundamental theorems (FFT and SFT)
of invariant theory for the  orthosymplectic supergroup.  The FFT in particular tells us that
the invariants of the orthosymplectic supergroup all arise,  in some appropriate sense,
from the bilinear form defining the supergroup. We continue the study here to develop the invariant theory
for the orthosymplectic Lie superalgebra.

It is well-known that the invariant theory of the orthogonal Lie algebra, or
special orthogonal group, is more intricate than that of the
orthogonal group.  Consider for example the actions of the complex orthogonal group ${\rm O}(m)$
and its Lie algebra $\fso(m)$ on the $m$-th tensor power of
$M=\C^m$. There exists an $\fso(m)$-module homomorphism
$\det: \C\longrightarrow  M^{\otimes m}$ where $\C$ is the trivial $\fso(m)$-module,
 given by the composition
$\C\stackrel{\sim}\longrightarrow \wedge^m M\hookrightarrow M^{\otimes m}$.
This determinant map is also an ${\rm O}(m)$-module homomorphism,
but $\C$ then needs to be thought of as the non-trivial $1$ dimensional ${\rm O}(m)$-representation.
The image of $\det$ is then not an ${\rm O}(m)$-invariant, but a pseudo ${\rm O}(m)$-invariant in the sense that
$g(\det)=\det(g) \det$ for any $g\in {\rm O}(m)$, where $\det(g)$ is the determinant of $g$
as an $m\times m$ matrix.  Then for all  $r(\geq m)$, the subspace of $\fso(m)$-invariants in $M^{\otimes r}$ is ``generated" by
the determinant map and the pull-back of the bilinear form defining ${\rm O}(m)$
(see e.g., \cite[Appendix F]{FH}).

We study here the analogous question for the orthosymplectic Lie superalgebra.
The main issue is to understand those invariants of
the orthosymplectic Lie superalgebra, which are not invariants of the orthosymplectic supergroup.
We show in this work that to generate all $\osp$-invariants,
just one invariant of $\osp$ is required, which is not an invariant of the orthosymplectic supergroup
$\OSp$. This is the super Pfaffan.
In references \cite{S1, S2}, Sergeev anticipated the existence of
a super Pfaffian, which is such an invariant. However, the super Pfaffian remained
somewhat mysterious, as the arguments for its existence in \cite{S1, S2}
were abstract. 

By using integration over the supergroup in the Hopf superalgebraic setting \cite{SZ01, SZ05},
we described in \cite{LZ6} a construction for the super Pfaffian, which is quite appealing
conceptually.   In principle the construction allows one to compute the super Pfaffian,
but even in small dimensions, it is quite nontrivial to derive an explicit expression for it
this way,  see Example \ref{eg:osp2-2} below. The aims of the present paper are to gain a better
understanding of the super Pfaffian, and to present a proof of its sufficiency for generation of
all invariants of $\osp$.

We now briefly describe the main results of the paper.

Let $V$ be a complex superspace with $\sdim(V)=(m|2n)$, regarded as the natural module for
the orthosymplectic Lie algebra $\osp(V)$ (see \cite[\S 2.4]{LZ6} for details),
and let $\cS(N)$ be the coordinate ring of the dual space of
$V^{\oplus N}\simeq V\ot_{\C}\C^N$,
the sum of $N$ copies of  $V$.
Then $\cS(N)$ admits commuting actions of $\osp(V)$ and $\gl_N$,
and hence the subalgebra  $\cS(N)^{\osp(V)}$ of $\osp(V)$-invariants is a $\gl_N$-module.
Theorems \ref{lem:hwv-even} and \ref{thm:pseudo-inv} describe the decomposition of
 $\cS(N)^{\osp(V)}$ as a $\gl_N$-module; specifically, we determine the simple submodules and
show that they all have unit multiplicity.  The highest weight vector
of each simple submodule is obtained.
These theorems imply Corollary \ref{cor:osp-gl}, which shows that
$\dim(V^\lambda)^{\osp(V)}\le 1$ for any simple tensorial $\gl(V)$- module $V^\lambda$,
and gives the necessary and sufficient conditions on the highest weight $\lambda$ for equality to hold.

We introduce the space $\Gamma(N)$ of super Pfaffians  in $\cS(N)^{\osp(V)}$
(see Definition \ref{def:Pfaffian-space});
this is a simple $\gl_N$-submodule.
Theorem \ref{thm:main1} proves the
formula \eqref{def:Omega} for the highest weight vector
of $\Gamma(N)$, which generates the entire space of super Pfaffians as $\gl_N$-module.
The formula is simple and explicit,  which renders transparent
the invariance properties of the super Pfaffians under the action of
the orthosymplectic Lie superalgebra. It plays a crucial role
in the proof of Theorem \ref{thm:generating-set}.

Theorem \ref{thm:generating-set} is the first fundamental theorem (FFT) of invariant theory for the
 orthosymplectic Lie superalgebra $\osp(V)$. It states that the super Pfaffians and
 the elementary invariants of the orthosymplectic supergroup
 (see Remark \ref{rem:elmt-invs}) generate all invariants of the
orthosymplectic Lie superalgebra acting on the superalgebra $\cS(N)$.
 This is analogous to the case of the orthogonal Lie algebra,
 but the super Pfaffians in the current context are more complicated than the
determinant.

The results on $\cS(N)$ may be interpreted in terms of the tensor powers $V^{\otimes N}$,
leading to a thorough understanding of the subspace $\left(V^{\otimes N}\right)^{\osp(V)}$ of
$\osp(V)$-invariants in $V^{\otimes N}$ for any $N$.
The key results are given in Theorem \ref{thm:tensor-invs} and Corollary  \ref{cor:tensor-pseudo-invs}.

There has been much interest in recent years in studying endomorphism algebras of Lie superalgebras and Lie supergroups,
and our results may also be interpreted in terms of these associative algebras of endomorphisms.
The FFT for $\OSp(V)$ asserts  \cite{LZ5} that the endomorphism
algebra $\End_{\OSp(V)}(V^{\otimes r})$ is a quotient of the Brauer algebra of degree $r$ with parameter $m-2n$,
where $\sdim(V)=(m|2n)$.  Now $\End_{\OSp(V)}(V^{\otimes r})$ is contained as a proper
subalgebra in $\End_{\osp(V)}(V^{\otimes r})$. One can deduce from Corollary \ref{cor:tensor-pseudo-invs}
the additional endomorphisms needed.
It would be very worthwhile to better understand the structure of $\End_{\osp(V)}(V^{\otimes r})$.

\section{Invariants of the orthosymplectic supergroup}\label{sect:ivariants}
\subsection{Preliminaries}
Let $V=\C^{m|2n}=\C^m\oplus\C^{2n}$,  where  $\C^m$ (resp. $\C^{2n}$) is the even (resp. odd) subspace.
Let $(e_1, e_2, \dots, e_m)$ and $(e_{m+1}, \dots, e_{m+2n})$ be the standard bases for $\C^m$ and $\C^n$ respectively
and call $B=(e_1, e_2, \dots, e_m; e_{m+1}, \dots, e_{m+2n})$  the standard basis of $V$.
Endow $V$ with an even non-degenerate supersymmetric bilinear form $(\ , \ )$, and write
$\kappa= (\kappa_{i j})$ with $(e_i, e_j)= \kappa_{i j}$ for all $i, j$.
We shall choose the basis $B$ so that
$\kappa= \begin{pmatrix} I_m & {\bf 0} \\ {\bf 0} & \eta\end{pmatrix}$
with $\eta=\begin{pmatrix} {\bf 0}& I_n & \\ -I_n & {\bf 0}\end{pmatrix}$ and $I_r$ being the $r\times r$ identity matrix.
{
Then we have the evaluation map $\hat{C}: V\ot V\longrightarrow \C$, $v\ot w \mapsto (v, w)$, and co-evaluation map
$
\check{C}: \C\longrightarrow V\ot V,
$
defined by
\[
(\id_V\ot \hat{C})(\check{C}(1)\ot v) = (\hat{C}\ot \id_V)(v\ot \check{C}(1))= v, \quad \forall v\in V.
\]
In terms of the basis elements, we have
\begin{eqnarray}\label{eq:rigidity}
\check{C}(1)=\sum_{a=1}^{m+2n} e_a\ot (\kappa^{-1})_{a b}  e_b.
\end{eqnarray}
}

Denote by $\osp(V)$ the orthosymplectic Lie superalgebra on $V$ defined
with respect to this bilinear form. The restriction of the bilinear form 
to $V_{\bar0}$ (resp. $V_{\bar1}$) is 
symmetric (resp. skew symmetric), thus 
we have the corresponding orthogonal group ${\rm O}(V_{\bar0})$ and symplectic group ${\rm Sp}(V_{\bar1})$.
Let  $\OSp(V)$ be 
the Harish-Chandra super pair
$(\OSp(V)_0, \osp(V))$, where $\OSp(V)_0= {\rm O}(V_{\bar0})\times {\rm Sp}(V_{\bar1})$.
Then $\OSp(V)$  can be regarded as the orthosymplectic supergroup, as explained in \cite{DM}.

{
It is clear that both the evaluation and co-evaluation maps are $\OSp(V)$-invariant.

Fix a positive integer $N$. Let $V^N=V\otimes\C^N$, where $\C^N$ is purely even.
Denote by $\cS(N)$ the coordinate ring of $(V^N)^*=\Hom_\C(V^N, \C)$. Then $\cS(N)$ is the
complex
symmetric  superalgebra over $V^N$ (cf. \cite[\S 3.2]{LZ6}), and we have
\[
\cS(N)=S((V^N)_{\bar0})\otimes\wedge((V^N)_{\bar1}).
\]
}
Clearly,
$\cS(N)=\underbrace{\cS(1)\otimes \dots \otimes \cS(1)}_N$. The
$\Z_+$-grading of $\cS(1)=\oplus_{k\ge 0}\cS^k(1)$ with $\cS^1(1)=V$ gives rise to a
$\Z_+^N$-grading for $\cS(N)=\oplus_{k_1, \dots, k_N\ge 0} \cS^{k_1, \dots, k_N}(N)$ with
$\cS^{k_1, \dots, k_N}(N)=\cS^{k_1}(1)\otimes\dots\otimes\cS^{k_N}(1)$.
This provides a $\Z_+$-grading
 $S(N)=\oplus_{k\ge 0}\cS^k(N)$ with $\cS^k(N)=\oplus_{k_1+\dots+k_N=k}\cS^{k_1, \dots, k_N}(N)$,
and we have a surjective homomorphism of $\Z_+^N$-graded associative superalgebras
\be\label{eq:pi}
\pi: \cS(N)\longrightarrow S((V^N)_{\bar0}),
\ee
which is the identity on $S((V^N)_{\bar0})$ and is the projection to the
constant term (or augmentation map) on $\wedge((V^N)_{\bar1})$.
For any $f\in\cS(N)$, we call $\pi(f)$ its {\em leading term}.

There are commuting actions of $\gl(V)$ and $\gl_N$ (and also of $\GL(V)$ and $GL_N$)
on $V^N=V\otimes \C^N$, which  naturally extend to $\cS(N)$.
A version of Howe duality \cite{H} (also see \cite{CW}) states that as a $\gl(V)\times\gl_N$- module,
\begin{eqnarray}\label{eq:Howe}
\cS(N)=\oplus_\mu V^\mu\otimes L_\mu,
\end{eqnarray}
where $L_\mu$ (resp. $V^\mu$) is the simple module for $\gl_N$ (resp. $\gl(V)$)
with highest weight $\mu=(\mu_1, \mu_2, \dots)$, and the sum is over
all partitions $\mu$ of length $\le N$ lying in the $(m, 2n)$-hook.
The length of a partition $\mu$ is the largest integer $k$ such that $\mu_k\ne 0$,
and $\mu$ is within the $(m, 2n)$-hook if
$\mu_{m+1}\le 2n$.
Note that the decomposition \eqref{eq:Howe} is multiplicity free, that is, the multiplicity of each simple
$\gl(V)\times\gl_N$-submodule is one.  This result will be used repeatedly.

The $\gl(V)$ (resp. $\GL(V)$) action restricts to an action of the orthosymplectic Lie superalgebra $\osp(V)$
(resp. orthosymplectic supergroup $\OSp(V)$) on $\cS(N)$.
It is easy to see that $S((V^N)_{\bar0})$ is an $\OSp(V)_0$-module with trivial
$\Sp((V)_{\bar1})$-action, and $\pi$ is an $\OSp(V)_0$-module homomorphism.

To explain our computations explicitly, we take cordinates as follows: let $f_1=(1\ 0 \ 0 \dots \ 0)$, $f_2=(0 \ 1\ 0 \dots \ 0)$,
$\dots$, $f_N=(0\ 0 \dots 0 \ 1)$ be the standard basis of $\C^N$.
Then $\cS(N)$ is the tensor product of the polynomial algebra in the even variables $x^i_t=e_i\otimes f_t$
($1\le i\le m$ and  $1\le t\le N$) and
the Grassmann (i.e. exterior) algebra generated by the odd variables
$\theta^\mu_t=e_{m+\mu}\otimes f_j$ for $1\le\mu\le 2n$, $1\le t\le N$.

\begin{remark}\label{rem:pi}
In this notation, the `leading term map' $\pi$ satisfies $\pi(x_t^i)=x_t^i$ and $\pi(\theta_t^\mu)=0$, for all $i,t,\mu$.
\end{remark}

\subsubsection{Some notation}\label{ss:not}
For $t=1,\dots,N$, write
$$
\X_t:=(X^1_t, \dots, X^{m+2n}_t)=(x^1_t, \dots, x^m_t, \theta^1_t, \dots, \theta^{2n}_t);
$$
this is an ordered basis of the subspace $V_t:=V\ot f_t$ of $V\ot_\C \C^N$.
%

{
Each $X^a_t$ defines a complex valued linear function on $(V^N)^*= V^*\ot(\C^N)^*$ in the standard way, namely
$X^a_t({\bar v}\ot{\bar f})
=(-1)^{[e_a][{\bar v}]}{\bar v}(e_a){\bar f}(f_t)$ for any ${\bar v}\in V^*$ and ${\bar f}\in (\C^N)^*$.
A monomial function $X_{t_1}^{a_1}X_{t_2}^{a_2}\dots X_{t_k}^{a_k}$ then satisfies
\[
X_{t_1}^{a_1}X_{t_2}^{a_2}\dots X_{t_k}^{a_k}({\bar v}\ot{\bar f})
=(-1)^{[e_{a_1}][{\bar v}]}  \prod_{i=1}^k {\bar v}(e_{a_i}){\bar f}(f_{t_i}).
\]
We remark that for any supercommutative superalgebra $\Lambda$ (e.g., the infinite Grassmann algebra in \cite{LZ6}),
$\Lambda\ot_\C\cS(N)$ may be regarded as the ring of polynomial functions  $(V^N)^*\longrightarrow \Lambda$.

Let $Q(N)=(q_{ij})$ be the $N\times N$ matrix with the $(t,s)$-entry given by
\be\label{eq:qij}
q_{ts}:=\X_t\kappa\inv(\X_s)^{\text{tr}},
\ee
where $\tr$ denotes the transpose of a matrix.
This definition is equivalent to
\be\label{eq:matqij}
Q(N)=\X\kappa\inv \X^\tr,
\ee
where $\X=\begin{pmatrix}
\X_1\\
\X_2\\
.\\
.\\
.\\
\X_N
\end{pmatrix}$.

For any ${\bar v}\ot{\bar f}\in (V^N)^*$, we have
$
q_{s t} ({\bar v}\ot{\bar f}) = ({\bar v}\otimes {\bar v})(\check{C}(1)) {\bar f}(f_s){\bar f}(f_t),
$
where
$\check{C}$ is the co-evaluation map and  $\check{C}(1)$ is given by \eqref{eq:rigidity}.
As $\check{C}(1)$ is $\OSp(V)$-invariant, $q_{s t}$ is an $\OSp(V)$-invariant in $\cS(N)$ for any $s,t$.
}

For $a=1,\dots,m+2n$, let $X_{a t}= \sum_{c=1}^{m+2n}(\kappa^{-1})_{a c} X^c_t$ and
$\partial_{a t}=\frac{\partial}{\partial X_t^a}$. With this notation, the actions of $\osp(V)$ and of $\gl_N$
on $\cS(N)$ may be described as follows.

The action of the orthosymplectic Lie superalgebra $\osp_{m|2n}$ on $\cS(N)$
is realised  by the differential operators $J_{ab}$, defined as follows \cite{Z08}:
\begin{eqnarray}\label{eq:diff-op}
J_{a b} = \sum_{t=1}^N\left(X_{a t}\partial_{b t}
- (-1)^{[a][b]} X_{b t}\partial_{a t}\right), \quad 1\le a, b\le m+2n,
\end{eqnarray}
where $[\  ]: \{1, 2, \dots, m+2n\}\longrightarrow \{0, 1\}$ with $[a]=0$ if $a\le m$
and $[a]=1$ if $a>m$.  Let $\{E_{i j}\mid 1\le i, j\le N\}$ be the standard basis of
$\gl_N$ with the commutation relations
$
[E_{i j}, E_{k \ell}]=\delta_{j k} E_{i \ell} -\delta_{\ell i} E_{k j}.
$
Then $\gl_N$ can be realised on $\cS(N)$ by
\begin{eqnarray}\label{eq:diff-gl}
E_{s t} = \sum_{a=1}^{m+2n} X_s^a\partial_{a t},  \quad  1\le s, t\le N.
\end{eqnarray}

Note that since the $E_{tt}$ are just Euler operators,
the subspaces $\cS^{k_1, \dots, k_N}(N)$, which are the
homogeneous components of $\cS(N)$ with respect to the $\Z_+^N$-grading of $\cS(N)$,
are the weight spaces of this $\gl_N$ module, and
the homogeneous $\Z_+$-graded subspaces are eigenspaces of
$E=\sum_{t=1}^N E_{t t}$. Evidently,  the $\gl_N$-action on $\cS(N)$ preserves the $\Z_+$-grading.

It is easy to show that the Laplacian operators
$
\partial^2_{s t} = \sum_{a, b=1}^{m+2n}\kappa_{a b} \partial_{b s} \partial_{a t}
$
commute with the generators $J_{a b}$ of $\osp(V)$.  Write
$(\partial_{\bar 0})^2_{s t} = \sum_{a, b=1}^{m}\kappa_{a b} \partial_{b s} \partial_{a t}$
and $(\partial_{\bar 1})^2_{s t} = \sum_{a, b=m+1}^{m+2n}\kappa_{a b} \partial_{b s} \partial_{a t}$.
Then $\partial^2_{s t} =(\partial_{\bar 0})^2_{s t} +(\partial_{\bar 1})^2_{s t}$.

\subsection{Invariants of the orthosymplectic supergroup}\label{subsect:ivariants}
In this section, we study the subalgebra of ${\OSp(V)}$-invariants in $\cS(N)$
from the above point of view.
Denote by $\cS(N)^{\osp(V)}$ the subalgebra of $\osp(V)$-invariants
in $\cS(N)$,
and by $\cS(N)^{\OSp(V)}$ that of the $\OSp(V)$-invariants.
Then $\cS(N)^{\osp(V)}$ is stable under the action of $\OSp(V)_0$.
Further, the kernel of  $\det:\OSp(V)_0\to\{\pm 1\}$, which is just $\SOSp(V)_0:=\SO(V_{\bar0})\times\Sp(V_{\bar1})$
 is contained in the enveloping algebra of the
$\osp(V)$-action on $\cS(N)$, and therefore acts trivially on $\cS(N)^{\osp(V)}$.

It follows \cite[Appendix F]{FH} that $\OSp(V)_0$ acts on the invariants of $\osp(V)$,
and we have
\begin{eqnarray}
\cS(N)^{\osp(V)}= \cS(N)^{\OSp(V)}\oplus \cS(N)^{\OSp(V), \det},
\end{eqnarray}
where
$
\cS(N)^{\OSp(V), \det} = \{f\in \cS(N)^{\osp(V)} \mid g(f)=\det(g) f, \ \forall g\in \OSp(V)_0\};
$
this space will be referred to as the space of pseudo invariants of $\OSp(V)$.
Note that both $\cS(N)^{\OSp(V)}$  and $\cS(N)^{\OSp(V), \det}$ are  semi-simple $\gl_N$-submodules of $\cS(N)$.

\begin{remark}\label{rem:elmt-invs} For $i, j=1, 2, \dots, N$, we have defined in \eqref{eq:qij} the quadratic
elements $q_{i j}$ of $\cS(N)$.
It follows from the FFT of invariant theory for the orthosymplectic supergroup \cite{LZ6}
that the subalgebra $\cS(N)^{\OSp(V)}$ is generated as superalgebra by these quadratic elements $q_{i j}$,
which will be referred to as the elementary $\OSp(V)$-invariants.
\end{remark}

It follows from the block nature of $\kappa$ and \eqref{eq:qij} that the quadratic $\OSp(V)$-invariants $q_{ij}$
may be expressed as $q_{ij}=p_{ij}+\phi_{ij}$, where

\begin{eqnarray}\label{eq:quadratics}
p_{i j} = \sum_{k=1}^m x^k_i  x^k_j, & \phi_{i j}=\sum\limits_{\mu, \nu=1}^{2n}
\theta_i^\mu(\eta^{-1})_{\mu \nu}\theta_j^\nu.
\end{eqnarray}
A simple computation shows that
$p_{i j}$ and $\phi_{i j}$ are $\OSp(V)_0$-invariant, and  that $J_{a b}(q_{i j})=0$ when $[a]=[b]$ (i.e. for the even operators $J_{ab}$).

Fix $N>0$, and for each $k$ with $1\le k\le N$, we form the $k\times k$-matrices
\[
\begin{aligned}
P(k)&=(p_{i j})_{i, j=1}^k, \quad
\Phi(k)&=(\phi_{i j})_{i, j=1}^k,  \quad
Q(k)&=P(k)+\Phi(k),
\end{aligned}
\]
and define $D(k)=\det Q(k)$.  Let $\tilde{Q}(k)$ be the adjugate matrix
of $Q(k)$, that is, $(-1)^{i+j}\tilde{Q}(k)_{j i}$ is the determinant of
the matrix obtained by removing the $i$-th row and $j$-th column from $Q(k)$.

\begin{lemma}
Assume that $1\le k\le N$ and $1\le j\le k$.  Then
\begin{eqnarray}
\partial_{a j}(D(k))  &=&  2  \sum_{r=1}^k  X_{a r} \tilde{Q}(k)_{r j}, \quad 1\le a\le m+2n, \label{eq:D-formulae-1} \\
\partial^2_{j j}(D(k)) &=&2(m-2n-k+1)\tilde{Q}_{j j},  \label{eq:D-formulae-2}
\end{eqnarray}
\begin{eqnarray}
\partial^2_{k k}(D(k)^\ell) &=& 2\ell(m-2n-k+2\ell-1)D(k)^{\ell-1}D(k-1), \label{eq:Laplace}
\end{eqnarray}
\begin{eqnarray}
E_{i j} (D(k))&=& 2  \sum_{r=1}^k q_{i r} \tilde{Q}(k)_{r j}, \quad 1\le i\le N,  \label{eq:ED}\\
D(k+1) &=& q_{k+1, k+1} D(k) - \frac{1}{2} \sum_{r=1}^k q_{r, k+1} E_{k+1, r}(D(k)). \label{eq:reduce}
\end{eqnarray}
\end{lemma}
\begin{proof}
The proof of these formulae is straightforward, so we provide a sketch, leaving details to the reader.
Equation \eqref{eq:D-formulae-1} is easily verified by direct computation.
Applying
$\sum_{b=1}^{m+2n}\kappa_{a b}\partial_{b j}$ to \eqref{eq:D-formulae-1} and summing over
$a$ from $1$ to $m+2n$, we obtain
\[
\partial^2_{j j}(D(k)) =2(m-2n)\tilde{Q}_{j j} + 2 \sum_{r=1}^k E_{r j}(\tilde{Q}(k)_{r j}).
\]
This leads to \eqref{eq:D-formulae-2} taking into account the following formulae:
\begin{eqnarray}\label{eq:EQ}
E_{j j}(\tilde{Q}(k)_{j j})=0, \quad E_{r j}(\tilde{Q}(k)_{r j})=-\tilde{Q}(k)_{j j}, \  \forall r\ne j.
\end{eqnarray}
Equation \eqref{eq:D-formulae-1} gives
$
\sum_{a, b=1}^{m+2n}\kappa_{b a}\partial_{a k}(D(k))  \partial_{b k}(D(k)) = 4D(k)D(k-1).
$
This together with \eqref{eq:D-formulae-2} yields \eqref{eq:Laplace}.
Equation \eqref{eq:ED} follows from \eqref{eq:D-formulae-1} immediately.
For any $r\le k$, we have
$
\tilde{Q}(k+1)_{r, k+1} =-\sum_{s=1}^k q_{k+1, s} \tilde{Q}(k)_{s r},
$
and it then follows from \eqref{eq:ED} that
$
\tilde{Q}(k+1)_{r, k+1}= -\frac{1}{2} E_{k+1, r}(D(k)).
$
Using this in the Laplace expansion of $D(k+1)$ along its last column,
we arrive at \eqref{eq:reduce}.
\end{proof}

It follows from \eqref{eq:ED} that
\[
\begin{aligned}
E_{i j}(D(k))=\left\{\begin{array}{l l l l}
0, 		& 1\le i<j\le N, \\
2 D(k),  	& 1\le i=j\le k, \\
0,       	& i=j> k.
\end{array}\right.
\end{aligned}
\]
Therefore,  each $D(k)$ is a $\gl_N$-highest weight vector with weight $2\omega_k$,
where $\omega_r=(\underbrace{1, \dots, 1}_r, 0, \dots, 0)$ is the $r^{\text th}$ fundamental weight of $\gl_N$.

Given any $N$-tuple of nonnegative integers $(\ell_1, \ell_2, \dots, \ell_N)$, we let
\[
\lambda=(\lambda_1, \dots, \lambda_N)=2\sum_{i=1}^N\ell_i\omega_i.
\]
It is clear that such $\lambda$ comprise all even partitions  of length $\le N$
(i.e., those where all rows have even length).  Define
\begin{eqnarray}\label{eq:Dlambda}
D_\lambda:=\prod_{i=1}^N D(i)^{\ell_i}.
\end{eqnarray}
Since the operators $E_{i j}$ are derivations of $\cS(N)$, the element $D_\lambda$,
if nonzero, is a highest weight vector of $\gl_N$ with weight $\lambda$. We next have

\begin{lemma}\label{lem:hw-vectors}
If $\lambda_{m+1}\le 2n$, then $D_\lambda\ne 0$.
\end{lemma}

{
\begin{proof} The superalgebra $\cS(N)$ is freely generated as supercommutative superalgebra by
the even generators $X^i_t$ ($1\leq i\leq m$) and the odd generators $X^i_t$ ($m+1\leq i\leq m+2n$).
We may therefore consider the homomorphism of supercommutative superalgebras
\begin{eqnarray}\label{eq:specialization-map}
R: \cS(N)\longrightarrow \Lambda_{N-m}:=\langle \{\theta_t^j \mid 1\leq j\leq 2n,\; m+1\leq t\leq N\} \rangle
\end{eqnarray}
defined by sending the generators, arranged in the matrix $\X$, (see \eqref{eq:matqij}) to the elements of the matrix
$$
\bordermatrix{&m&2n&\cr
                m&I_m & 0\cr
                N-m& 0 &  \Theta\cr},
$$
where
$$
\Theta=\begin{pmatrix}
\theta_{m+1}^1&.&.&.&\theta_{m+1}^{2n}\\
\theta_{m+2}^1&.&.&.&\theta_{m+2}^{2n}\\
&.&.&.&\\
&.&.&.&\\
\theta_{N}^1&.&.&.&\theta_{N}^{2n}\\
\end{pmatrix}.
$$

To investigate the image of the matrix $Q(k)$ under $R$,
set
\begin{eqnarray}\label{eq:Psi}
\Psi(s)=\begin{pmatrix}
\phi_{m+1, m+1} & \phi_{m+1, m+2}&\dots & \phi_{m+1, s} \\
\phi_{m+2, m+1} & \phi_{m+2, m+2}&\dots & \phi_{m+2, s} \\
\dots&\dots&\dots&\dots\\
\phi_{s, m+1} & \phi_{s, m+2}&\dots & \phi_{s s}
\end{pmatrix}.
\end{eqnarray}
Then
$R(Q(k))=I_k$ if $k\le m$, and
$R(Q(s))= \begin{pmatrix}I_m & 0 \\ 0 & \Psi(s)\end{pmatrix}$
if $s>m$. In particular, $R(D(k))=1$ for $k\le m$.
We will prove the theorem by examining properties of $D_\lambda$ under the map $R$.

Let  $d$ be the length of $\lambda$ (i.e., the largest $i$ such that $\ell_i>0$).
If $d\le m$, we have $R(D_\lambda)=1$.
If $d>m$, let us consider $\nabla^{(\lambda)}(D_\lambda)$, where
\begin{eqnarray}\label{eq:laplace-power}
\nabla^{(\lambda)}=\prod_{k=m+1}^N\left(\partial_{k k}^2\right)^{\lambda_k/2}.
\end{eqnarray}
By iterating \eqref{eq:Laplace}, we obtain the formula
\begin{eqnarray}\label{eq:truncation}
(\partial^2_{k k})^\ell(D(k)^\ell) =C(k, \ell) D(k-1)^{\ell},
\end{eqnarray}
where
\begin{eqnarray}\label{eq:coeff}
C(k, \ell) = \prod_{j=1}^\ell 2j(m-2n-k+2j-1),
\quad \text{with \ \ }
C(k, 0)=1.
\end{eqnarray}
This leads to
\begin{eqnarray}\label{eq:truncating-D}
\nabla^{(\lambda)}(D_\lambda)= C(\lambda) D(m)^{\frac{\lambda_m}{2}}\prod_{i=1}^{m-1} D(i)^{\ell_i},  \  \ C(\lambda)=\prod_{k=m+1}^N C(k, \frac{\lambda_k}{2}).
\end{eqnarray}
We note crucially that
$C(k, \ell)\ne 0$,  if $k>m$ and $\ell\le n$.
Thus
\begin{eqnarray}\label{eq:coeff-C}
C(\lambda)\ne 0,  \  \text{as $\lambda_k/2\le \lambda_m/2\le n$ for all $k>m$}.
\end{eqnarray}
Now $R(\nabla^{(\lambda)}(D_\lambda))= C(\lambda)$, and hence $D_\lambda\ne 0$.

This completes the proof.
\end{proof}
}

With the help of Lemma \ref{lem:hw-vectors}, we can decompose $\cS(N)^{\OSp(V)}$
with respect to the $\gl_N$-action.
\begin{theorem}\label{lem:hwv-even}
The subalgebra of $\OSp(V)$-invariants in $\cS(N)$ has the following multiplicity free decomposition
as a $\gl_N$-module
\[
\cS(N)^{\OSp(V)}=\bigoplus_\lambda L_\lambda,
\]
where $L_\lambda$ is the simple $\gl_N$-module with highest weight $\lambda$,
and the sum is over all partitions $\lambda$ of length $\le N$ all of whose parts are even and satisfy
the condition $\lambda_{m+1}\le 2n$.  Furthermore, for each $\lambda$, the element
$D_{\lambda}$ defined in \eqref{eq:Dlambda} is a highest weight vector of
$L_\lambda$.
\end{theorem}
\begin{proof} It follows from \eqref{eq:Howe} that
$
\cS(N)^{\OSp(V)}=\oplus_\mu (V^\mu)^{\OSp(V)}\otimes L_\mu.
$
By the FFT for ${\OSp(V)}$ \cite{LZ6}, $\cS(N)^{\OSp(V)}$
is generated as superalgebra by the elementary invariants $q_{i j}$ ($1\le i, j\le N$).
Then the linear span of the elementary invariants  is a simple $\gl_N$-module isomorphic to $S^2(\C^N)$.
Thus $\cS(N)^{\OSp(V)}$ is a quotient of
$S(S^2(\C^N) )$, the symmetric algebra over $S^2(\C^N)$.
It is well known (see, e.g., \cite[\S 7.8, Example (c)]{KP}) that
$S(S^2(\C^N) )$  is
multiplicity free and contains all the simple modules associated
with even partitions of length $\le N$.
But it follows from Lemma \ref{lem:hw-vectors} that if $\mu$ is an even partition
such that $V^\mu$ appears as a submodule of $\cS(N)$,
then $\dim  (V^\mu)^{\OSp(V)}=1$. This proves the first assertion, while
the second one has already been proved.
\end{proof}

Having understood the invariants of $\OSp(V)$, we now turn to the study of the space
of pseudo invariants of the orthosymplectic supergroup.

\section{Constructions of a super Pfaffian} \label{sect:pfaffian}

Throughout this section, we take $N$ of the last section to be equal to $m$, and write $\cS=\cS(m)$,
$P=P(m)$, $Q=Q(m)$ and $\Phi=\Phi(m)$.   In this case, we have the
$m\times m$-matrix $X=(x^i_j)$, and it is clear that $P=X^t X$.
Let $\Delta=\det X$.

\subsection{First definition of a super Pfaffian}

Let $\cS[\Delta^{-1}]$ be the localisation  of $\cS$ at $\Delta$.
The leading term homomorphism $\pi$ (see \eqref{eq:pi}, Remark \ref{rem:pi}) extends uniquely to a superalgebra homomorphism
$
\pi: \cS[\Delta^{-1}]\longrightarrow S((V^m)_{\bar0})[\Delta^{-1}],
$
which we still denote by $\pi$, and which respects the $\OSp(V)_0$-action defined below.
\begin{lemma}
The $\OSp(V)$ action on $\cS$ extends uniquely to the localisation $\cS[\Delta^{-1}]$. The action is described explicitly as follows.
Let $Y\in \osp(V)$, $g\in \OSp(V)_0$ and  $\frac{f}{\Delta^k}\in\cS[\Delta\inv]$, with $f\in\cS$. Then
\begin{eqnarray}
&&Y\left(\frac{f}{\Delta^k}\right)= \frac{Y(f)}{\Delta^k}- k  \frac{Y(\Delta)f}{\Delta^{k+1}}, \label{eq:act1}\\
&&g\left(\frac{f}{\Delta^k}\right)=\det(g)^{-k} \frac{g(f)}{\Delta^k}.  \label{eq:act2}
\end{eqnarray}
\end{lemma}
\begin{proof}
The elements of the Lie superalgebra $\osp(V)$ as realised by \eqref{eq:diff-op} act naturally on
$\cS[\Delta^{-1}]$ as $\Z_2$-graded derivations, and formula \eqref{eq:act1} follows.
The extended $\OSp(V)_0$-action must satisfy $g(fh)=g(f)g(h)$, for $f,h\in\cS[\Delta\inv]$. Thus
\[
g(f)=g\left(\Delta^k \frac{f}{\Delta^k}\right)=g(\Delta^k ) g\left(\frac{f}{\Delta^k}\right)=\det(g)^k \Delta^k g\left(\frac{f}{\Delta^k}\right),
\]
which yields \eqref{eq:act2}.
\end{proof}

Since $\pi(\det(Q))=\Delta^2$,  $\pi\left(\frac{\det Q}{\Delta^2}\right)=1$. Hence there is an element
$\zeta\in \ker(\pi)$ such that
$
\det Q = \Delta^2 +\zeta,
$
and $\zeta$ is nilpotent.  Therefore,  we may define an element $F\in\cS[\Delta^{-1}]$
such that $F^2=\frac{\det Q}{\Delta^2}$ as a {\em Taylor polynomial} in $\frac{\zeta}{\Delta^2}$:
\begin{eqnarray}\label{eq:F}
\begin{aligned}
F:=& \sqrt\frac{\det Q}{\Delta^2}=
\sqrt{1+\frac{\zeta}{\Delta^2}}.
\end{aligned}
\end{eqnarray}
An elementary computation leads to
\[
F=1+\sum_{k\ge 1} \frac{1}{k!}\prod_{j=0}^{k-1} \left(\frac{1}{2}-j\right)
\cdot \left(\frac{\zeta}{\Delta^2}\right)^{k}.
\]
\begin{lemma}\label{lem:sqrt-inv}
The element $\Delta F$ is a pseudo invariant for $\OSp(V)$. Specifically,
for any $Y\in\osp(V)$ and $g\in\OSp(V)_0$,  we have
\[
Y(\Delta F)=0, \text{ and }\quad g(\Delta F)=\det(g) \Delta F.
\]
\end{lemma}
\begin{proof}
Since $\det Q$ is in  $\cS^{\OSp(V)}$,
it follows from the $\Z_2$-graded derivation property
of any $Y\in\osp(V)$ that
$
Y((\Delta F)^2)= (2\Delta F )Y(\Delta F)=Y(\det Q)=0;
$
since $\Delta F$ is invertible in $\cS[\Delta\inv]$, it follows that
$Y(\Delta F)=0$ in $\cS[\Delta^{-1}]$, i.e. that $\Delta F$ is $\osp(V)$-invariant.

>From the formula for $F$ given above, clearly $gF=F$.  It follows from \eqref{eq:act2}
that $g(\Delta F)= \det(g) \Delta F$ for any $g\in \OSp(V)_0$.
\end{proof}

We note that $\Delta F\not\in\cS$ unless $n=0$. In fact,
\begin{lemma}\label{lem:regular-singular}
If $k<n$, then $\Delta F (\det Q)^k$ does not belong to $\cS$.
However, there exists a smallest positive integer $N_0$ such that $\Delta F (\det Q)^{N_0+j}$
belongs to $\cS$ for all $j\in\Z_{\geq 0}$.
\end{lemma}

\begin{proof}
Let $W$ be the superspace $\C^{1|2n}$ with basis $t,\theta^1,\dots,\theta^{2n}$, where $t$ is even
and the  $\theta^\mu$ are arbitrary Grassmann variables.
Consider the surjective homomorphism of superalgebras $\xi:\cS\lr S(W)$, where the free generators
$X^i_t$ of $\cS$ are mapped as follows.

With $X_i$ as in \ref{ss:not}, we have
$\xi(X_i) = (0, \dots, 0, \underbrace{1}_i, 0, \dots, 0)$ for $i< m$ and
$\xi(X_m) = (0, \dots, 0, t,  \theta^1 \dots, \theta^{2n})$.

Then
$\xi(Q)=\begin{pmatrix}
1&0&...\\
0&1&0&...\\
&&...&&\\
&&&&&1&0\\
&&&&&0&t^2+\nu
\end{pmatrix},$
where $\nu=\xi(\phi_{m,m})$.

Thus $\xi(\Delta)=t$, and we may extend $\xi$ to $\xi:\cS[\Delta\inv]\lr S(W)[t\inv]$.

Evidently, $\xi(\det Q)= t^2 + \nu$.  Moreover a simple calculation shows that
$$
\xi(\phi_{m,m})=\nu=2\sum_{j=1}^n\theta^{2j-1}\theta^{2j}.
$$

Hence
\be\label{eq:exp}
\begin{aligned}
\xi(\Delta F (\det Q)^k)=&t^{2k+1} \left(1 + \sum_{\ell\ge 1}\frac{1}{\ell!}
\prod_{j=0}^{\ell-1}\left(\frac{1}{2}+k-j\right)\cdot
\left(\frac{\nu}{t^2}\right)^\ell\right).
\end{aligned}
\ee
Now $S(W)[t\inv]\simeq\C[t,t\inv]\ot\Lambda(\theta^1,\dots,\theta^{2n})$, where $\Lambda=\Lambda(\theta^1,\dots,\theta^{2n})$
is the exterior algebra on $\sum_j\C\theta^j$. Thus each of its elements  has a unique expansion $\sum_{i\in\Z}t^i\xi_i$,
with $\xi_i\in\Lambda$.
But from the right hand side of equation \eqref{eq:exp} the coefficient of $t^{2(k-n)+1}$ in the
expansion of $\xi(\Delta F (\det Q)^k)$ is
$$
 \nu^n  \frac{1}{n!}\prod_{j=0}^{n-1}\left(\frac{1}{2}+k-j\right).
$$
Moreover from the explicit formula for $\nu$ given above, $\nu^n\neq 0$.
Hence if $k<n$, $\xi(\Delta F (\det Q)^k)$ is not in $S(W)$.
It follows that $\Delta F (\det Q)^k\not\in\cS$,  proving the first statement.

As for the second, note that since $\zeta=\det Q - \Delta^2$ is nilpotent, there is a positive integer $N$ such that $\zeta^{N+1}=0$.
Then
\[
\Delta F (\det Q)^N=\Delta^{2N+1}+ \sum_{k\ge 1} \frac{1}{k!}\prod_{j=0}^{k-1} \left(\frac{1}{2}+N-j\right)\cdot  \Delta^{2(N-k)+1}{\zeta}^k.
\]
If $k\le N$, then $\Delta^{2(N-k)+1}\zeta^k$
belongs to $\cS$; if $k> N$, we have $\Delta^{2(N-k)+1}\zeta^k=0$.
Hence $\Delta F (\det Q)^N\in\cS$. This proves the existence of the desired $N_0$.
\end{proof}
We shall see  in Section \ref{sect:construct}  that the $N_0$ in Lemma \ref{lem:regular-singular} is equal to $n$.

\begin{definition}
Define the element $\Omega$ by
\begin{eqnarray}\label{def:Omega}
\Omega=\Delta F (\det Q)^n.
\end{eqnarray}
This $\Omega$ will be referred to as a super Pfaffian.
\end{definition}
For the moment, we know only that $\Omega\in\cS[\Delta\inv]$, but
subject to the above assertion about $N_0$, we now have the following result.
\begin{theorem} \label{thm:main1} The super Pfaffian $\Omega$
belongs to $\cS^{\osp(V)}$. It is even, homogeneous of degree
$m(2n+1)$, and has leading term $\pi(\Omega)=\Delta^{2n+1}$.  Furthermore,
$\Omega$ satisfies
$
g(\Omega)=\det(g)\Omega
$
for any $g\in \OSp(V)_0$, hence is a pseudo $\OSp(V)$-invariant.
\end{theorem}

\begin{proof}
All the statements of the theorem follow directly from the definition of $\Omega$
and Lemma \ref{lem:sqrt-inv}, except the claim that
$\Omega\in\cS$.  This will be shown after  we establish
Theorem \ref{thm:equivalence}.
\end{proof}

\begin{remark}
When $n=0$, we have $V=\C^m$, and $\osp(V)$ and $\OSp(V)$ become
$\fso(m)$ and ${\rm O}(m)$ respectively.  In this case $\Omega$ reduces to $\Delta$.
Thus the super Pfaffian is a generalisation of the determinant.
\end{remark}

\begin{example}
If $m=1$,  we have $q= x^2 + \phi$ with $\phi:=\sum_{\mu, \nu=1}^{2n}\theta^\mu (\kappa^{-1})_{\mu\nu} \theta^\nu$.
The super Pfaffian is given by
\[
\begin{aligned}
\Omega =& x^{2n+1}+ \sum_{\ell=1}^n\frac{1}{\ell!}\prod_{j=0}^{\ell-1}\left(\frac{1}{2}+n-j\right)\cdot
 x^{2(n-\ell)+1} \phi^\ell.
\end{aligned}
\]
Note that this is indeed an element of $\cS$.
\end{example}

\subsection{A second construction of the super Pfaffian}\label{sect:construct}
We review the construction of a pseudo invariant given in \cite{LZ6},
and show that it coincides with $\Omega$ up to a sign.

We write $G=\OSp(V)$, $G_0=\OSp(V)_0$, and
$\fg=\osp(V)$, and denote by
${\rm U}(\fg)$ the universal enveloping superalgebra of $\fg$.
Any unexplained notation appearing in this subsection
can be found in \cite{LZ6,SZ01, SZ05}.

Let $\U(\fg)^0$ be the finite dual Hopf superalgebra of
$\U(\fg)$. Any locally finite $\U(\fg)$-module $M$ may be regarded as a right $\U(\fg)^0$-comodule
with the structure map $\omega: M\longrightarrow M\otimes \U(\fg)^0$,
$m\mapsto \sum_{(m)} m_{(0)}\otimes m_{(1)}$,
defined by
\[
\sum_{(m)} m_{(0)}  \langle m_{(1)}, x\rangle = (-1)^{[m][x]}x m, \quad \forall x\in \U(\fg).
\]
Here $m,m_0\in M$, $m_1\in\U(\fg)^0$, and $\langle\:,\;\rangle$ is the pairing between $\U(\fg)^0$
and $\U(\fg)$.

It was shown in \cite{SZ05} (also see \cite{SZ01}) that there exists a
nontrivial left and right $\fg$-invariant integral
$
\int:  \U(\fg)^0 \longrightarrow \C.
$
Let $\cI_M:=(\id_M\otimes\int)\omega$. Then $\cI_M(M)\subseteq M^\fg$.

Next observe that in the $\gl(V)\times\gl_N$ Howe duality \eqref{eq:Howe} on $\cS(N)$,
the simple highest weight $\gl(V)$-module $V^\lambda$ may be taken to be defined with respect to the
standard Borel subalgebra $\fb$ of $\gl(V)$ spanned by the matrix units $E_{a b}$
($1\le a\le b\le m+2n$)  relative to the basis $B$.
Given $V^\lambda\subset\cS(N)$, let
\[
V^\lambda_0=\{v\in V^\lambda\mid \gl(V)_{+1}(v)=0 \}, \quad
V^\lambda_b=\{v\in V^\lambda\mid \gl(V)_{-1}(v)=0 \};
\]
these are
respectively the top and bottom components of
$V^\lambda$ with respect to the $\Z$-grading of $\U(\gl(V))$ defined by $deg(\gl(V)_{\pm 1})=\pm 1$.
Both subspaces are $\gl(V)_{\bar0}$-modules,
so we may consider their  respective subspaces
 $(V^\lambda_0)^{\osp(V)_{\bar0}}$ and $(V^\lambda_b)^{\osp(V)_{\bar0}}$ of $\osp(V)_{\bar0}$-invariants.

The following result is \cite[Lemma 7.2]{LZ6} stated slightly more generally. Its proof
remains the same. Let $\cI_{\cS(N)}$ be the map $\cI_M$ defined earlier for $M=\cS(N)$.
\begin{lemma}\label{lem:construct}
\begin{enumerate}
\item
Let $V^\lambda\subset\cS(N)$ be a simple $\gl(V)$-submodule which is typical, and assume that
there exists a nonzero $\delta_0$ that belongs to  either $(V^\lambda_0)^{\osp(V)_{\bar0}}$ or
$(V^\lambda_b)^{\osp(V)_{\bar0}}$.   Then $\delta:=\cI_{\cS(N)}(\delta_0)$ is a nonzero element
of $(V^\lambda)^{\osp(V)}\subset\cS(N)^{\osp(V)}$.
\item If the $\osp(V)_{\bar0}$-invariant $\delta_0$
satisfies $g(\delta_0)=\det(g) \delta_0$
for all $g\in \OSp(V)_0$, then $\delta$ is a pseudo $\OSp(V)$-invariant satisfying
$g(\delta)=\det(g) \delta$.
\end{enumerate}
\end{lemma}

Now we return to the special case $\cS=\cS(m)$.
Fix some order of the elements $\theta^\mu_j$, and form the product
$\Pi=\prod_{\mu, j} \theta^\mu_j$ according to the chosen order. Then $\Pi$ is uniquely defined up to a sign.
It is in the top degree component of the Grassmann algebra.

\begin{theorem} \label{thm:pfaff} Let $\delta_0=\Delta\Pi$.
 Then $\delta_0$ is a lowest weight vector, which satisfies the properties required by both parts of
Lemma \ref{lem:construct}, and hence $\cI_\cS(\delta_0)$ is
a nonzero pseudo invariant of $\OSp(V)$.
\end{theorem}

The integral on $\U(\fg)^0$ can be described more explicitly \cite{SZ01, SZ05}.
Let $p: \U(\fg)^0\longrightarrow \U(\fg_{\bar0})^0$ be the restriction map induced by the natural
Hopf superalgebra embedding $\U(\fg_{\bar0})\longrightarrow\U(\fg)$.  There exists a unique (left and right)
$\fg_{\bar0}$-invariant integral $\int_0: \U(\fg_{\bar0})^0\longrightarrow \C$ such that $\int_0 1=1$,
thus we have the composition map $\int_0 p: \U(\fg)^0\longrightarrow\C$,
which can be regarded as an element of $(\U(\fg)^0)^*$.
Consider $Z=\U(\fg)/\U(\fg)\fg_{\bar0}$ as a $\fg$-module, and let $Z^\fg$ be its invariant submodule.
It was shown in \cite{SZ05} (and in \cite{SZ01} for various cases)
that $\dim Z^\fg = 1$.  Fix a generator $z+\U(\fg)\fg_{\bar0}$ of $Z^\fg$.  Let
$\nu: \U(\fg)\longrightarrow(\U(\fg)^0)^*$ be the superalgebra embedding
defined for any $x\in\U(\fg)$ by $\nu(x)(a)= (-1)^{[a][x]}a(x)$ for all $a\in\U(\fg)^0$. Then
$
\int:= \nu(z)\int_0 p,
$
which is independent of the representative $z$ chosen for
$z+\U(\fg)\fg_{\bar0}$. This leads to the following formula (see \cite[Definition 7.3]{LZ5})
\begin{eqnarray}
\cI_\cS(\delta_0)=z(\Delta \Pi).
\end{eqnarray}

Let us take a PBW basis for $\U(\fg)$ such that generators of $\fg_{\bar0}$ appear to the right
of those of $\fg_{\bar1}$.
Then the representative $z$ of $z+\U(\fg)\fg_{\bar0}$ can be expressed as a sum of
basis elements involving only elements of $\fg_{\bar1}$.
We regard $z$ as in $\U(\gl(V))$ via the embedding $\fg\subset\gl(V)$,
and decompose it into a sum of elements which are weight vectors with respect to the adjoint action
of $\gl(V)_{\bar0}\cong \gl_m(\C)\oplus\gl_{2n}(\C)$. If we denote by $z_+$ the element with the highest weight
in the decomposition of $z$ as a sum of weight vectors,
we have $z_+\in\wedge^{2mn} \gl(V)_{\bar1}$,
where $\gl(V)_{\bar 1}$ is the span of $E_{i, m+\nu}$ with $1\le i\le m$ and $1\le \nu\le 2n$.
Then $z_+$ is equal to the product of all $E_{i, m+\nu}$
in any given order, up to a nonzero scalar multiple $c$. We shall choose $z$ so as
to make $c=\pm 1$.  Note that $g z_+= \det(g)^{2n} z_+$ for all $g\in\GL((V)_{\bar0})\times\GL((V)_{\bar1})$.
The leading term $\pi(\cI_\cS(\delta_0))$ of $\cI_\cS(\delta_0)$ is $z_+(\delta_0)$, which is a highest weight vector of the
simple $\gl(V)$-submodule of $\cS$ containing $\delta_0$. Up to a sign,  it is given by
\begin{eqnarray}\label{rm:leading-term}
\pi(\cI_\cS(\delta_0))=\Delta^{1+2n}.
\end{eqnarray}
Note that $\cI_\cS(\delta_0)$ is even and homogeneous of degree $m(2n+1)$.

\begin{example}\label{eg:osp2-2}  We consider $\cI_\cS(\delta_0)$ for the Lie superalgebra $\osp_{2|2}$.
Now $\cS=\C[x^i_{t}, \theta^\mu_t\mid i, \mu, t=1, 2]$, where $x^i_t$ are ordinary variables
and $\theta^\mu_t$ are Grassmann variables. Then $\osp_{2|2}$ is realised by the following differential operators on $\cS$:
\[
\begin{aligned}
&J^{1 2} =\sum_{t=1}^2\left(x^1_t\frac{\partial}{\partial x^2_t}-x^2_t\frac{\partial}{\partial x^1_t}\right), \\
&J^{i 3}=\sum_{t=1}^2\left(x^i_t\frac{\partial}{\partial \theta^2_t} - \theta^1_t \frac{\partial}{\partial x^i_t}\right), \quad
J^{i 4}=\sum_{t=1}^2\left(-x^i_t\frac{\partial}{\partial \theta^1_t} - \theta^2_t \frac{\partial}{\partial x^i_t}\right), \quad i=1, 2,\\
&J^{3 3}=\sum_{t=1}^2\theta^1_t\frac{\partial}{\partial \theta^2_t}, \quad
J^{4 4}=-\sum_{t=1}^2\theta^2_t\frac{\partial}{\partial \theta^1_t}, \quad
J^{3 4}=\sum_{t=1}^2\left(-\theta^1_t\frac{\partial}{\partial \theta^1_t} + \theta^2_t\frac{\partial}{\partial \theta^2_t}\right).
\end{aligned}
\]

It is much more convenient to work with the variables $z_t=x^1_t+\sqrt{-1}x^2_t$ and  $\bar{z}_t=x^1_t-\sqrt{-1}x^2_t$.
We have $J^{1 2} =\sqrt{-1}\sum_{t=1}^2\left(z_t\frac{\partial}{\partial z_t}-{\bar z}_t\frac{\partial}{\partial {\bar z}_t}\right)$.
Let $J^\alpha=J^{1 \alpha}+\sqrt{-1} J^{2 \alpha}$ and ${\bar J}^\alpha=J^{1 \alpha}-\sqrt{-1} J^{2 \alpha}$, then
\[
\begin{aligned}
&J^3=\sum_{t=1}^2\left(z_t\frac{\partial}{\partial \theta^2_t} - 2\theta^1_t \frac{\partial}{\partial {\bar z}_t}\right), \quad
J^4=\sum_{t=1}^2\left(-z_t\frac{\partial}{\partial \theta^1_t} - 2\theta^2_t \frac{\partial}{\partial {\bar z}_t}\right),\\
&{\bar J}^3=\sum_{t=1}^2\left({\bar z}_t\frac{\partial}{\partial \theta^2_t} - 2\theta^1_t \frac{\partial}{\partial  z_t}\right), \quad
{\bar J}^4=\sum_{t=1}^2\left(-{\bar z}_t\frac{\partial}{\partial \theta^1_t} - 2\theta^2_t \frac{\partial}{\partial z_t}\right).
\end{aligned}
\]
In view of \cite[Theorem 4]{SZ01}, we let
$\gamma= J^3 J^4 {\bar J}^4 {\bar J}^3$,  then
$\cI_\cS(\delta_0)=-\gamma(\delta_0)$ with
$\delta_0=(z_1{\bar z}_2 - {\bar z}_1 z_2)\theta^1_1\theta^2_1\theta^1_2\theta^2_2$.

We now work out the explicit expression of $\cI_\cS(\delta_0)$ as an element of $\cS$.
Let $\Delta=z_1{\bar z}_2 - {\bar z}_1 z_2$,
$\Pi_i=\theta^i_1\theta^i_2$. Then $\delta_0=-\Delta\Pi_1\Pi_2$.  We have
\[
\begin{aligned}
&{\bar J}^3(\delta_0)= \sum_i {\bar z}_i\frac{\partial}{\partial \theta^2_i}\delta_0, \quad
{\bar J}^4{\bar J}^3(\delta_0)  =- \sum_{i, j} {\bar z}_j {\bar z}_i \frac{\partial}{\partial \theta^1_j}
								\frac{\partial}{\partial \theta^2_i}\delta_0,\\
&J^4{\bar J}^4{\bar J}^3(\delta_0)= \Delta^2 \sum_j {\bar z}_j \frac{\partial}{\partial \theta^2_j}\Pi_2
								- 8\sum_j {\bar z}_j \frac{\partial}{\partial \theta^1_j}\delta_0,
\end{aligned}
\]
and $-\cI_\cS(\delta_0)=J^3 J^4{\bar J}^4{\bar J}^3(\delta_0)=\Omega_0+\Omega'_2+\Omega''_2+\Omega_4$, with
\[
\begin{aligned}
&\Omega_0=\Delta^2 \sum_{i, j}  z_i  {\bar z}_j\frac{\partial}{\partial \theta^2_i}  \frac{\partial}{\partial \theta^2_j}\Pi_2,
&\quad  \Omega_4=16 \sum_{i, j} \theta^1_i  \frac{\partial}{\partial{\bar z}_i} \left({\bar z}_j\frac{\partial}{\partial \theta^1_j}\delta_0\right), \\
&\Omega'_2= - 2 \sum_{i, j} \theta^1_i  \frac{\partial}{\partial{\bar z}_i} \left(\Delta^2  {\bar z}_j \frac{\partial}{\partial \theta^2_j}\Pi_2 \right), &\quad
 \Omega''_2=-8\sum_{i, j} z_i  \frac{\partial}{\partial \theta^2_i}\left({\bar z_j}  \frac{\partial}{\partial \theta^1_j}\delta_0\right).
\end{aligned}
\]
We have
\begin{eqnarray}\label{eq:omega04}
\Omega_0=- \Delta^3, \quad \Omega_4=48\delta_0.
\end{eqnarray}
It is a matter of putting up with the pain of doing the tedious and lengthy computations to show that
\[
\begin{aligned}
\Omega'_2= &-2\Delta^2(\theta^1_1\theta^2_2-\theta^1_2\theta^2_1)
-4\Delta (z_1{\bar z}_1\theta^1_2\theta^2_2+z_2{\bar z}_2\theta^1_1\theta^2_1)\\
			&+4\Delta (z_1{\bar z}_2\theta^1_2\theta^2_1-z_2{\bar z}_1\theta^1_1\theta^2_2),\\
\Omega''_2= &-8\Delta (z_1{\bar z}_1\theta^1_2\theta^2_2+z_2{\bar z}_2\theta^1_1\theta^2_1)
			+ 8\Delta (z_1{\bar z}_2\theta^1_1\theta^2_2+z_2{\bar z}_1\theta^1_2\theta^2_1) .
\end{aligned}
\]
Let $\Omega_2:=\Omega'_2+\Omega''_2$. Then
\begin{eqnarray}\label{eq:omega2}
\begin{aligned}
\Omega_2=&-12\Delta (z_1{\bar z}_1\theta^1_2\theta^2_2+z_2{\bar z}_2\theta^1_1\theta^2_1)
			+6 \Delta (z_1{\bar z}_2  +z_2{\bar z}_1) (\theta^1_1\theta^2_2+\theta^1_2\theta^2_1).
\end{aligned}
\end{eqnarray}
Combining \eqref{eq:omega04} and \eqref{eq:omega2}, we obtain
\begin{eqnarray}\label{eq:omega}
\begin{aligned}
\cI_\cS(\delta_0)=&\Delta^3 +12\Delta (z_1{\bar z}_1\theta^1_2\theta^2_2+z_2{\bar z}_2\theta^1_1\theta^2_1) \\
			&-6 \Delta (z_1{\bar z}_2  +z_2{\bar z}_1) (\theta^1_1\theta^2_2+\theta^1_2\theta^2_1)  - 48\delta_0.
\end{aligned}
\end{eqnarray}
Note that all the four terms on the right hand side  are invariant with respect to the even subalgebra $\fg_{\bar0}=\mathfrak{so}_2\oplus\mathfrak{sp}_2$. In fact,
$z_1{\bar z}_1$, $\theta^1_1\theta^2_1$, $z_2{\bar z}_2$, $\theta^1_2\theta^2_2$, $z_1{\bar z}_2  +z_2{\bar z}_1$
and $\theta^1_1\theta^2_2+\theta^1_2\theta^2_1$ are all $\fg_{\bar0}$-invariant.
\end{example}

%
%
The following result relates Theorem \ref{thm:pfaff} to the super Pfaffian $\Omega$.
\begin{theorem}\label{thm:equivalence}
 The element $\cI_\cS(\delta_0)\in \cS$ defined in Theorem \ref{thm:pfaff}
coincides with the super Pfaffian $\Omega$ up to a sign.
\end{theorem}
\begin{proof}
By the first fundamental theorem of invariant theory for the orthosymplectic supergroup $G$ proved in \cite{LZ6},
the subalgebra $\cS^G$ of $G$-invariants is generated by the elements $q_{i j}$ $(i\le j)$.
Let $\langle p_{i j} \mid i\le j\rangle$ be the subalgebra of
$\cS$ generated by $p_{i j}$ with $i\le j$.  Note that
this in fact is the polynomial algebra $\C[p_{i j}\mid i\le j]$, as can be seen in many ways , e.g.,
by using the second fundamental theorem of invariant theory for ${\rm O}(V_{\bar0})$. We have
$\pi(\cS^G)= \C[p_{i j}\mid i\le j]$.
This implies that
$\cS^G$ is the polynomial algebra $\C[q_{i j} \mid i\le j]$,
and we have the algebra isomorphism
\[
\iota=\pi|_{\cS^G}: \cS^G\longrightarrow \C[p_{i j}\mid i\le j], \quad q_{i j} \mapsto p_{i j}.
\]

Consider $\widetilde{\Omega}=\cI_\cS(\delta_0)$ given in Theorem \ref{thm:pfaff}. We have
$\widetilde{\Omega}^2\in\C[q_{i j} \mid i\le j]$, and by \eqref{rm:leading-term},
$
\iota(\widetilde{\Omega}^2)=(\Delta^2)^{1+2n}= (\det P)^{1+2n}.
$
Thus we conclude that $\widetilde{\Omega}^2=(\det Q)^{1+2n}=\Omega^2$.  Therefore,
$\left(\frac{\widetilde{\Omega}}{\Omega}\right)^2=1$ in $\cS[\Delta^{-1}]$, and this implies
that $\frac{\widetilde{\Omega}}{\Omega}$ is either $1$ or $-1$. Hence
$\widetilde{\Omega}$ is equal to $\Omega=\Delta F (\det Q)^n$  up to sign, completing the proof.
\end{proof}

\begin{proof}[Proof of Theorem \ref{thm:main1}]
Since  $\cI_\cS(\delta_0)$ belongs to $\cS$ by construction, and coincides with $\Omega$ up to a sign
by the above theorem, we immediately arrive at Theorem \ref{thm:main1}.
\end{proof}

The following theorem will be generalised to $\cS(N)$ for arbitrary $N$ in Theorem \ref{thm:main2}.
Its proof is much easier than that of the general case.
\begin{theorem} \label{thm:critical} Let  $\cS=\cS(m)=S(V\ot\C^m)$.
The  subalgebra $\cS^{\osp(V)}$ of $\osp(V)$-invariants in $\cS$ is generated by
$\Omega$ and the elementary invariants
$q_{i j}$ $(i, j=1, 2, \dots m)$.
\end{theorem}
\begin{proof}
Consider a highest weight vector $A\in \cS^{\OSp(V), \det}$. Then by the invariant theory of
$SO(V_{\bar0})$, there exists $D_\lambda$ such that $\pi(A)=\Delta\pi(D_\lambda)$. Thus
\[
\pi(\Omega A)= \Delta^2 \pi(D_{\lambda+2n \omega_m})=\pi(D_{\lambda+(2n+1) \omega_m}),
\]
and hence $\Omega A = D_{\lambda+(2n+1) \omega_m}$. This leads to
$
A= \Delta F D_\lambda.
$
Now apply the arguments in the proof of part (2) of Lemma \ref{lem:regular-singular} to $A$.
Setting $p_{i j}=\delta_{i j}$ for $i, j<m$,  and $p_{k m}=p_{m k} =t^2\delta_{k m}$ for all $k$,
we easily see that $A$ is non-singular as $t\to 0$ only if  $D_\lambda$ satisfies
$\ell_m\ge n$. In this case, $A=\Omega D_\mu$ with $\mu_i=\lambda_i-2n$.
Since $\Omega$ is an $\gl(V_{\bar0})$-invariant, every vector in the $\gl(V_{\bar0})$-submodule generated
by $A=\Omega D_\mu$ is of the form $\Omega f$ for some $f$ in the
submodule generated by $D_\mu$.   This completes the proof.
\end{proof}

It follows that $\cS^{\osp(V)}=\cS^{\OSp(V)}\oplus \Omega\cS^{\OSp(V)}$ as an $\cS^{\OSp(V)}$-module.

\section{Pseudo invariants of the orthosymplectic supergroup}

The following result is an easy consequence of the invariant theory of
the orthogonal and symplectic Lie algebras \cite{FH}.
\begin{lemma}\label{lem:small-k}
If $k<m$, then $\cS(k)^{\osp(V)}=\cS(k)^{\OSp(V)}$.
\end{lemma}
\begin{proof}
By the invariant theory of $\osp(V)_{\bar0}={\mathfrak{so}}(V_{\bar0})\times {\mathfrak{sp}}(V_{\bar1})$ \cite{FH},
the subspace  of $\osp(V)_{\bar0}$-invariants in $\cS(k)$ for $k<m$ is
generated by the elements $p_{ij}$ and $\phi_{ij}$.
It follows that there exist no pseudo $\OSp(V)$-invariants in this case,
and the lemma is proved.
\end{proof}

\subsection{Decomposition of the space of pseudo invariants}
Now we fix $N\ge m$. Then the super Pfaffian $\Omega=\Delta F D(m)^n$ belongs to $\cS(N)$,
and is a highest weight vector of $\gl_N$ with weight $(2n+1)\omega_m$. Let
$\Gamma(N)$ be the simple $\gl_N$-submodule of $\cS(N)$ generated by $\Omega$;
this is a direct summand in $\cS(N)$ and is of  dimension
$ \prod_{j=m+1}^N \prod_{i=1}^m\frac{2n+1+j-i}{j-i}$.
\begin{definition}\label{def:Pfaffian-space}
The module $\Gamma(N)$ will be called the {\it space of super Pfaffians}, and
each weight vector of  $\Gamma(N)$ will be referred to as a super Pfaffian.
\end{definition}
It is an immediate consequence of Theorem \ref{thm:main1} that $\Gamma(N)$ consists of
pseudo invariants of the orthosymplectic supergroup, that is, it is a subspace of  $\cS(N)^{\OSp(V), \det}$.
Hence from the Howe decomposition \eqref{eq:Howe}, we have
\[
\cS(N)^{\OSp(V), \det}=\oplus_\lambda  (V^\lambda)^{\OSp(V), \det}\otimes L_\lambda.
\]
\begin{theorem}\label{thm:pseudo-inv}
For any $N\ge m$,  the   subspace of pseudo $\OSp(V)$-invariants in $\cS(N)$
has the following multiplicity free decomposition as a $\gl_N$-module:
\begin{eqnarray*}
\cS(N)^{\OSp(V), \det}=\bigoplus_\lambda L_\lambda,
\end{eqnarray*}
where the sum is over all weights $\lambda=\omega_m+\sum_{i=1}^N 2\ell_i \omega_i$
with $\ell_i\in\Z_+$ such that $\lambda_{m+1}\le 2n$ and $\lambda_m\ge 2n+1$.
\end{theorem}

\begin{proof} The Howe duality decomposition \eqref{eq:Howe} requires that
$\lambda$ be a partition of length $\le N$ within the $(m, 2n)$-hook.
Let $A$ be a $\gl_N$ highest weight vector in $\cS(N)^{\OSp(V), \det}$.
Then $\Omega  A$ is a highest weight vector in $\cS(N)^{\OSp(V)}$, and
by weight considerations, it must be equal to a nonzero scalar multiple of  $D(m)^{n+1} D_\mu$ for some
\begin{eqnarray}\label{eq:Dmu}
D_\mu=\prod_{i=1}^N D(i)^{j_i}, \quad  j_i\in\Z_+.
\end{eqnarray}
Hence $A= c\Delta F D_\mu$ for some scalar $c\ne 0$.
This shows that $\cS(N)^{\OSp(V), \det}$ is multiplicity free as a $\gl_N$-module.

The weight $\lambda$ of $A$ automatically satisfies all the conditions of the theorem, except for
$\lambda_m=\mu_m+1\ge 2n+1$.  We claim that if $\mu_m< 2n$, then $A$ can not be an element of $\cS(N)$.
{
To prove this, we apply the method used in the proof of Lemma \ref{lem:hw-vectors}. Using
\eqref{eq:truncating-D}, we obtain $\nabla^{(\mu)}(\Delta F D_\mu)=C(\mu) \tilde{A}$ with
\[
\tilde{A}=\Delta F  D(m)^{\ell_m}\prod_{i=1}^{m-1} D(i)^{j_i},
\]
where $\ell_m=\mu_m/2$.
The scalar $C(\mu)$ is nonzero by reasoning like that which leads to \eqref{eq:coeff-C}.
The claim now follows if we can show that $\tilde{A}\not\in\cS(N)$.

To see this, consider the superalgebra homomorphism $\xi:\cS[\Delta\inv]\lr S(W)[t\inv]$
introduced in the proof of Lemma \ref{lem:regular-singular}.
It yields
\[\xi(\Delta F) = (t^2 +\phi_{m m})^{1/2}=t \left(1+\frac{\phi_{m m}}{t^2}\right)^{1/2},\]
where the square root is understood as a Taylor polynomial in $\frac{\phi_{m m}}{t^2}$.
We now have
\begin{eqnarray}
\xi(\tilde{A}) =(t^2 +\phi_{m m})^{1/2+\ell_m}=t^{1+\mu_m} \left(1+\frac{\phi_{m m}}{t^2}\right)^{1/2+\ell_m}.
\end{eqnarray}
The expansion of $\xi(\tilde{A})$ in $S(W)[t\inv]\simeq\oplus_{i\in\Z}t^i\ot\langle\theta^\mu_j\rangle$
includes the nonzero term $t^{1+\mu_m}  (\frac{\phi_{m m}}{t^2})^n=t^{1-2n+2\ell_m}(\phi_{m m})^n$,
which does not belong to $S(W)$ if $\mu_n< 2n$.

This proves the necessity of the condition $\mu_m\ge 2n$.
}

If $\lambda=(\lambda_1,  \lambda_2, \dots, \lambda_N)$ satisfies all the given conditions in the
theorem, $V^\lambda$ is typical  as a $\gl(V)$-module as $\lambda_m\ge 2n+1$.
Consider its top component $(V^\lambda)^0$ as a  $\gl(V)_{\bar0}$-module.
It is isomorphic to the tensor product of the simple $\gl(V_{\bar0})$-module with highest weight
$\mu=(\lambda_1,  \lambda_2, \dots, \lambda_m)$
and the simple $\gl(V_{\bar1})$-module with highest weight $\nu$,
where $\nu$ is the transpose partition of $(\lambda_{m+1},  \lambda_2, \dots, \lambda_N)$.
Here $\mu=(1, 1, \dots, 1)+\text{even partition}$, and all columns of $\nu$ are even.
Using the invariant theory of $\osp(V)_{\bar0}=\mathfrak{so}(V_{\bar0})\times \mathfrak{sp}(V_{\bar1})$, we see that
$(V^\lambda)^0$ contains a $1$-dimensional subspace of $\osp(V)_{\bar0}$-invariants,
which are pseudo invariants of $\OSp(V)_0$.
Now invoking  Lemma \ref{lem:construct}, we see that $(V^\lambda)^{\OSp(V), \det}\ne 0$.
This completes the proof.
\end{proof}

The proof  above also establishes the following fact.
\begin{corollary}\label{cor-eq:hwvs}
The set
$
\Xi:=\{\Delta F D_\mu\mid \mu_m\ge 2n\ge \mu_{m+1}\}
$
contains a nonzero highest weight vector of every simple $\gl_N$-submodule of
$\cS(N)^{\OSp(V),\det}$.
\end{corollary}

\subsection{Generators of $\osp(V)$-invariants}\label{sect:generators}
Finally we show that the space $\Gamma(N)$ of super Pfaffians together with the elementary invariants
$q_{i j}$ ($i, j=1, 2, \dots, N$) generate all the $\osp(V)$-invariants in $\cS(N)$.

\begin{theorem}\label{thm:main2}\label{thm:generating-set}
As a $\cS(N)^{\OSp(V)}$-module, $\cS(N)^{\OSp(V), \det}=\Gamma(N)\cS(N)^{\OSp(V)}$.
\end{theorem}

Let us make some preparations for the proof of this theorem.
We claim that the highest weight vectors in
Corollary \ref{cor-eq:hwvs} can all be expressed in terms of elements of $\Gamma(N)$ and $q_{i j}$.
This claim clearly implies the theorem.
Let \[
\Xi':=\{\Delta F D_\mu\mid \mu_1=\dots=\mu_m=2n\ge \mu_{m+1}\}.
\]
The following fact is straightforward.
\begin{fact}\label{fact1}
For any $A\in\Xi$ with weight $\lambda$ such that $\lambda_1>2n+1$,
there exist a product $D$ of powers of $D(i)$ with $i\le m$, and an element $B$ of $\Xi'$ such that $A=B D$.
\end{fact}
Thus our problem reduces to proving the above claim for the elements of $\Xi'$.
This will be done by further reducing the problem to one in the skew invariant theory of $\mathfrak{sp}(V_{\bar1})$,
that is, invariant theory in the setting of Grassmann algebras $\wedge(V_{\bar1}\otimes \C^s)$.
The main facts of   the theory are collected in \cite[\S3]{T}, which can also be easily deduced from Theorem \ref{lem:hwv-even}.
Consider $\cS(N)$ for $N=s+m$, and recall the specialisation homomorphism $R: \cS(N)\longrightarrow \Lambda_s$ defined
by \eqref{eq:specialization-map}.  The following result is a special case of Theorem \ref{lem:hwv-even}.
\begin{corollary}\label{thm:sp}
The subalgebra $(\Lambda_s)^{\mathfrak{sp}(V_{\bar1})}$ of invariants
is generated by the elements $\phi_{i j}=R(q_{i j})$ with $m+1\le i, j\le N$. As a $\gl_s$-module,
\[
 (\Lambda_s)^{\mathfrak{sp}(V_{\bar1})}=\bigoplus_\lambda \tL_\lambda,
\]
where $\tL_\lambda$ is the simple $\gl_s$-module with highest weight $\lambda$, and the sum is over all  even partitions
$\lambda$ of lengths $\le s$ satisfying the condition $\lambda_1\le 2n$.  Furthermore, the highest weight vector of
$\tL_\lambda$ is $\prod_{i=1}^s R(D(m+i))^{\ell_i}$ with $\lambda=2\sum_{i=1}^s\ell_i\omega_{i}$.
\end{corollary}

Denote by $L(B)$ the simple $\gl_N$-module generated by the highest weight vector
$B\in\Xi'$, and let $Q$ be the liner span of the elements $q_{i j}$ ($1\le i, j\le N$).  Then $L(B)Q$ is a semi-simple $\gl_N$-module. We let $Tr(L(B)Q)$ be the minimal
submodule which contains every simple submodule of $L(B)Q$ that is isomorphic to some $L(A)$ with $A\in\Xi'$.
Recall from \eqref{eq:Psi} that the matrix $\Psi(N)$ has entries $\phi_{i j}$ ($m+1\le i, j\le N$).
We shall also denote the span of its entries by $\Psi(N)$. 

For any $A\in\Xi'$, its specialisation $R(A)\in \Lambda_s$ is an $\mathfrak{sp}(V_{\bar1})$ highest weight vector. We denote 
 by $\tL(R(A))$ the $\mathfrak{sp}(V_{\bar1})$-submodule of $\Lambda_s$ generated by $R(A)$. 

\begin{fact}\label{fact2}
For any $A, B\in\Xi'$, if $\tL(R(A))\subset \tL(R(B))\Psi(N)$, then $L(A)\subset Tr(L(B)Q)$.
\end{fact}
\begin{proof}
Consider the set of simple $\gl_N$-submodules in $L(B)\otimes Q$ defined by
\[
K=\{ L_\lambda\subset L(B)\otimes Q \mid  \lambda_1=2n+1, \ \lambda_{m+1}\le 2n\}
\]
and the set of simple  $\gl_s$-submodules in $\tL(R(B))\otimes \Psi(N)$ given by
\[
\tilde{K}=\{ \tL_\nu\subset\tL(R(B))\otimes \Psi(N) \mid \nu_1\le 2n\}.
\]
It follows from the Littlewood-Richardson rules for $\gl_N$ and $\gl_s$ that $K$ and 
$\tilde{K}$ are in canonical bijection via the map
$L_\lambda\mapsto\tL_\nu$ with $\nu=(\lambda_{m+1}, \dots, \lambda_N)$.
The highest weight vector of any simple module in $K$ is of the form
\[
C:=B\otimes q_{k k} + \sum_{1\le i<k} b_i E_{k i}(B)\otimes q_{i k} + \sum_{1\le i\le j<k} b_{i j} E_{k i}E_{k j}(B)\otimes q_{i j}
\]
for some $k>m$, and scalars $b_i$ and $b_{i j}$ (depending on $k$). 
The corresponding simple $\gl_s$-module in $\tilde{K}$ has highest weight vector
\[
\tilde{C}:=R(B)\otimes \phi_{k k} + \sum_{m+1\le i<k} b_i E_{k i}(R(B))\otimes \phi_{i k} +
\sum_{m+1\le i\le j<k} b_{i j} E_{k i}E_{k j}(R(B))\otimes \phi_{i j}.
\]

Let 
\[
\begin{aligned}
&A:=B q_{k k} + \sum_{1\le i<k} b_i E_{k i}(B) q_{i k} + \sum_{1\le i\le j<k} b_{i j} E_{k i}E_{k j}(B) q_{i j}, \\
&\tilde{A}:=R(B) \phi_{k k} + \sum_{m+1\le i<k} b_i E_{k i}(R(B)) \phi_{i k} +
\sum_{m+1\le i\le j<k} b_{i j} E_{k i}E_{k j}(R(B)) \phi_{i j}.
\end{aligned}
\]
that is,  $A=M(C)$ and $\tilde{A}=M_{\Lambda_s}(\tilde{C})$, where
$M$ and $M_{\Lambda_s}$ denote the products in $\cS(N)$ and $\Lambda_s$ respectively
of the highest weight vectors.
Then $\tilde{A}=R(A)$ as all $R(q_{i r})=0$ if $i\le m< r$. Hence $A\ne 0$ if  $\tilde{A}\ne 0$.
\end{proof}

With the above preparations, we now prove the theorem.

\begin{proof}[Proof of Theorem \ref{thm:generating-set}]

It remains only to show that the highest weight vectors given by
Corollary \ref{cor-eq:hwvs} can all be expressed in terms of elements of $\Gamma(N)$ and $q_{i j}$.
We will do this by induction on the size of the highest weights, where the size of a partition
$\mu=(\mu_1, \mu_2, \dots )$ is $|\mu|=\sum_{i} \mu_i$.

The size of any weight of $\Xi$ is not less than $m(2n+1)$. If the size of the weight is
$m(2n+1)$, the corresponding vector in $\Xi$ is $\Omega$ (see Definition \ref{def:Omega}).
 If the size of the weight is $m(2n+2)+2$,
the corresponding vector in $\Xi'$ is $A=\Delta F D(m+1) D(m)^{n-1}$.
Using \eqref{eq:reduce}, we obtain
\[
A= q_{m+1, m+1} \Omega - \frac{1}{2n+1} \sum_{r=1}^m q_{r, m+1} E_{m+1, r}(\Omega) \in \Gamma(N) Q.
\]
There is also one vector in $\Xi\backslash\Xi'$ with a weight of size $m(2n+1)+2$, which
is in  $\Gamma(N)\cS(N)^{\OSp(V)}$ by  Assertion \ref{fact1}. This starts the induction.

For any integer $\ell\ge 0$, let  $W_{2\ell}$  be the homogeneous subspace of $\cS(N)^{\OSp(V), \det}$ of degree
$m(2n+1)+2\ell$. Similarly, let  $\tilde{W}_{2\ell}$ be the homogeneous subspace of
$(\Lambda_s)^{\mathfrak{sp}(V_{\bar1})}$ of degree $2\ell$.
Then $\tilde{W}_{2\ell+2}=\tilde{W}_{2\ell}\Psi(N)$ since it follows from Corollary \ref{thm:sp}
that  $(\Lambda_s)^{\mathfrak{sp}(V_{\bar1})}$ is generated by $\Psi(N)$.
By Assertion \ref{fact2},  corresponding to each $\gl_s$-highest weight vector $\tilde{A}$ in $\tilde{W}_{2\ell+2}$
with weight $\nu=(\nu_1, \nu_2, \dots, \nu_s)$ (which must be even),
there exists a $\gl_N$- highest weight vector $A\in W_{2\ell}Q$ such that $\tilde{A}=R(A)$, where the weight of $A$
is given by $\lambda=(\underbrace{2n+1, \dots, 2n+1}_m, \nu_1, \nu_2, \dots, \nu_s)$.  By Corollary \ref{thm:sp}
and Theorem \ref{thm:pseudo-inv}, these
weights $\lambda$ exhaust all the weights  of $\Xi'$ of size $m(2n+1)+2\ell+2$, and so
all vectors of $\Xi'$ with weights of this size are contained in $W_{2\ell}Q$. By Assertion \ref{fact1},
all elements in $\Xi\backslash\Xi'$  which have weight of this size are also in $W_{2\ell}Q$, thus all $\gl_N$ highest
weight vectors in $W_{2\ell+2}$ belong to $W_{2\ell}Q$. This implies $W_{2\ell}Q=W_{2\ell+2}$.
By the induction hypothesis, all vectors in $\Xi$ with weight of
size $m(2n+1)+2\ell$ are contained in $\Gamma(N)\cS(N)^{\OSp(V)}$. Hence
$W_{2\ell}$ is a subspace of $\Gamma(N)\cS(N)^{\OSp(V)}$,  and so also is $W_{2\ell+2}$.
\end{proof}

\section{Invariants of $\osp(V)$ in tensor powers of $V$}

The above results on $\osp(V)$-invariants in $\cS(N)$ enable one to understand the
$\osp(V)$-invariants in all tensor powers of $V$.
Recall that the $\Z_+^N$-gradation of $\cS(N)$ corresponds to the decomposition
into common eigenspaces of all the operators $E_{r r}$ with $1\le r \le N$ defined by \eqref{eq:diff-gl}, that is,
the weight spaces  in $\cS(N)$ of the $\gl_N$-algebra.
In particular, the space of weight $(1, 1, \dots, 1)$ will be called the zero $\gl_N$- weight space, which is given by
\begin{eqnarray}\label{eq:zerowt}
\cS^{(1, 1, \dots, 1)}(N)=V^{\otimes N}.
\end{eqnarray}

Denote by $D_\mu$ the zero $\gl_N$-weight space of the simple $\gl_N$-module $L_\mu$.
By considering the zero $\gl_N$-weight spaces of \eqref{eq:Howe}, we obtain the following multiplicity free decomposition
\begin{eqnarray}\label{eq:Schur}
\cS^{(1, 1, \dots, 1)}(N)=V^{\otimes N} =\oplus_\mu V^\mu\otimes D_\mu.
\end{eqnarray}

Recall the Schur-Weyl duality between the general linear superalgebra $\gl(V)$
and the symmetric group $\Sym_N$ on $V^{\otimes N}$. By decomposing
$V^{\otimes N}$ into direct sum of simple $\U(\gl(V))\otimes\C\Sym_N$-submodules,
and comparing the decomposition with \eqref{eq:Schur}, we see that
$D_\mu$ is the simple $\Sym_N$-module associated with the partition $\mu$ of $N$,
and the sum in \eqref{eq:Schur} is over all such $\mu$ that $\mu_{m+1}\le 2n$.

\begin{remark}
This implies the known fact \cite{Ko}  that
the zero $\gl_N$-weight space of the simple $\gl_N$-module $L_\mu$ forms a
simple $\Sym_N$-module isomorphic to that associated with the partition $\mu$ of $N$.
\end{remark}

Consider the $\osp(V)$-invariants on the right hand side of \eqref{eq:Schur}.
 The following result is an immediate consequence of Theorems \ref{lem:hwv-even} and \ref{thm:pseudo-inv} .
\begin{theorem} \label{thm:tensor-invs}
As a $\C\Sym_N$-module, the subspace $\left(V^{\otimes N}\right)^{\osp(V)}$ of $\osp(V)$-invariants in $V^{\otimes N}$
decomposes into the direct sum of two submodules,
\[
\left(V^{\otimes N}\right)^{\osp(V)} = \left(V^{\otimes N}\right)^{\OSp(V)}\oplus \left(V^{\otimes N}\right)^{\OSp(V), \det},
\]
which have the following multiplicity free decompositions:
\begin{enumerate}
\item
$
\left(V^{\otimes N}\right)^{\OSp(V)}=\bigoplus_\mu D_\mu ,
$
where the sum is over all  partitions $\mu $ of $N$ which are even and satisfy $\mu_{m+1}\le 2n$;  and
\item
$
\left(V^{\otimes N}\right)^{\OSp(V), \det}=\bigoplus_\lambda D_\lambda,
$
where the sum is over all partitions $\lambda$ of $N$ of the form $\lambda=\omega_m+\sum_{i=1}^N 2\ell_i \omega_i$
with $\ell_i\in\Z_+$ such that
$\lambda_{m+1}\le 2n<\lambda_m$.
\end{enumerate}
\end{theorem}

There is another description of $\left(V^{\otimes N}\right)^{\OSp(V)}$ which arises
from the tensor version of the FFT.
Let $\check{C}$ be the canonical invariant in $V\otimes V$, that is, if we write
$\check{C}=\sum_i v_i\otimes \bar v_i$, then $\sum_i v_i (\bar v_i, w)=w$ for all
$w\in V$.
The FFT for $\OSp(V)$ \cite{LZ6} may be stated as follows.
For any $N\in\Z_+$,
\begin{eqnarray}\label{eq:FFT-tensor}
(V^{\otimes N})^{\OSp(V)}=\left\{\begin{array}{l l}
\C\Sym_N(\check{C}^{\otimes\frac{N}{2}}) & \text{if $N$ is even}, \\
0 & \text{otherwise}.
\end{array}
\right.
\end{eqnarray}

Call any simple $\gl(V)$-submodule of $\oplus_{r\ge 0} V^{\otimes r}$ a
simple tensor module.
The following result is clear from the above theorem.
\begin{corollary}\label{cor:osp-gl}
Let $V^\mu$ be a simple tensor module for $\gl(V)$. Then $\dim  (V^\mu)^{\osp(V)}\le 1$, and equality holds if and only if
\begin{itemize}
\item either $\lambda$ is an even partition satisfying $\lambda_{m+1}\le 2n$;
\item or $\lambda=(\underbrace{1, \dots, 1}_m, 0, 0, \dots)+\mu$
for some even partition $\mu$ such that $\mu_{m+1}\le 2n\le \mu_m$.
\end{itemize}
\end{corollary}

Recall that in Definition \ref{def:Pfaffian-space} we define the space of super Pfaffians $\Gamma(N)$ for any $N\ge m(2n+1)$.
\begin{definition}
Let $r_c=m(2n+1)$, and denote by $\Gamma^0$ the zero $\gl_{r_c}$-weight space of $\Gamma=\Gamma(r_c)$.
Call $\Gamma^0$ the space of super Pfaffians in $V^{\otimes r_c}$.
\end{definition}
\begin{lemma}
The space $\Gamma^0$ of super Pfaffians is the simple $\Sym_{r_c}$-submodule in $V^{\otimes r_c}$ associated with the partition
of rectangular shape with $m$ rows and $2n+1$ columns.
\end{lemma}
We have the following result.

\begin{theorem} \label{thm:tensor-pseudo-invs}
\begin{enumerate}
\item If $N<r_c$， then $\left(V^{\otimes N}\right)^{\OSp(V), \det}=0$.
\item If $N\ge r_c$, then $\left(V^{\otimes N}\right)^{\OSp(V), \det}=
\C\Sym_N\left(\Gamma^0\otimes (V^{\otimes (N-r_c)})^{\OSp(V)}\right)$.
\end{enumerate}
\end{theorem}
\begin{proof}
This is immediate from  Theorem \ref{thm:main2}. Note that
part (1) follows more directly from the condition $\lambda_m>2n$ in Theorem \ref{thm:tensor-invs}.
\end{proof}

Let $\hat{C}_{i j}: V^{\otimes r}\longrightarrow V^{\otimes(r-2)}$
be the map given by contracting the $i$-th and $j$-th copies in
$V^{\otimes r}$ using the bilinear form on $V$. For example,
$\hat{C}_{1 2}: v_1\otimes v_2\otimes v_3\otimes \dots\otimes v_r\mapsto
( v_1,  v_2) v_3\otimes \dots\otimes v_r$.  Clearly all $\hat{C}_{i j}$  ($1\le i<j\le r$)
are $\OSp(V)$-module homomorphisms.
The {\em harmonic subspace} of $V^{\otimes r}$ is the intersection of the kernels of all $\hat{C}_{i j}$.

\begin{corollary}\label{cor:tensor-pseudo-invs}
We have $\left(V^{\otimes N}\right)^{\OSp(V), \det}\ne 0$ if and only if $N-r_c$ is a non-negative even integer,
and in this case
\[
\left(V^{\otimes N}\right)^{\OSp(V), \det}=
\C\Sym_N\left(\Gamma^0\otimes\check{C}^{\otimes\frac{N-r_c}{2}}\right).
\]
Furthermore, $\Gamma^0$ is contained in the harmonic subspace of $V^{\otimes r_c}$.
\end{corollary}
\begin{proof}
The first part of the corollary follows from part (2) of Theorem \ref{thm:tensor-pseudo-invs} and equation \eqref{eq:FFT-tensor}.
For any $i\ne  j$, we have $\hat{C}_{i j}(\Gamma^0)\subset \left(V^{\otimes (r_c-2)}\right)^{\OSp(V), \det}=0$
by part (1) of Theorem \ref{thm:tensor-pseudo-invs}.
Hence $\Gamma^0$ is harmonic.
\end{proof}




\begin{thebibliography}{9999}
\bibitem{FH} Fulton, William; Harris, {\em Joe Representation theory. A first course}.
Graduate Texts in Mathematics, {\bf 129}. Readings in Mathematics. Springer-Verlag, New York, 1991.

\bibitem{CW} S.-J. Cheng, W. Wang,  ``Howe duality for Lie superalgebras",  {\sl Compositio Math. \bf 128}
             (2001), no. 1, 55--94.

\bibitem{DM} Deligne, Pierre; Morgan, John W. ``Notes on supersymmetry (following Joseph Bernstein)".
Quantum fields and strings: a course for mathematicians, Vol. {\bf 1}, {\bf 2}
(Princeton, NJ, 1996/1997), 41--97, Amer. Math. Soc., Providence, RI, 1999.

\bibitem{H} R. Howe, ``Remarks on classical invariant theory",
{\sl Trans. Amer. Math. Soc. \bf 313} (1989), no. 2, 539--570.

\bibitem{Ko} B. Kostant,
``On Macdonald's $\eta$-function formula, the Laplacian and generalized exponents". {\sl Adv. Math. \bf 20} (1976) 257-285.

\bibitem{KP}  H. Kraft and C. Procesi, ``Classical invariant theory:  a primer".

\bibitem{LZ4}  G. I. Lehrer and R. B. Zhang, ``The second fundamental theorem of invariant theory for the orthogonal group".
{\sl Ann. of Math. (2) \bf 176} (2012), no. 3, 2031--2054.

\bibitem{LZ5} G. I. Lehrer and R. B. Zhang, ``The Brauer Category and Invariant Theory",
{\sl J. European Math Society}, in press;   arXiv:1207.5889 [math.GR].

\bibitem{LZ6} G. I. Lehrer and R. B. Zhang,
``The first fundamental theorem of invariant theory for the orthosymplectic supergroup",
 arXiv:1401.7395.

\bibitem{LZ7} G. I. Lehrer and R. B. Zhang,  ``The second fundamental theorem of invariant theory
for the orthosymplectic supergroup",  arXiv:1407.1058.

\bibitem{SZ01} Scheunert, M.; Zhang, R. B.
``Invariant integration on classical and quantum Lie supergroups", {\sl J. Math. Phys. \bf 42}  (2001), no. 8, 3871--3897.

\bibitem{SZ05} Scheunert, M.; Zhang, R. B.
``Integration on Lie supergroups: a Hopf algebra
approach", {\sl  J. Algebra \bf 292} (2005), 324--342.



\bibitem{S1}A. Sergeev,  ``An analog of the classical invariant theory for Lie superalgebras. I",
    {\sl Michigan Math. J. \bf 49} (2001), Issue 1, 113-146.

\bibitem{S2}A. Sergeev,  ``An analog of the classical invariant theory for Lie superalgebras. II",
    {\sl Michigan Math. J. \bf 49} (2001), Issue 1, 147-168.

\bibitem{T} George Thompson, ``Skew invariant theory of symplectic groups, pluri-Hodge groups and 3-manifold invariants".
{\sl Int. Math. Res. Not. IMRN 2007, no. 15, Art. ID rnm048, 32 pp}.

\bibitem{Z08} R. B. Zhang, ``Orthosymplectic Lie superalgebras in superspace analogues of quantum Kepler problems",
{\sl  Commu. Math. Physics \bf 280} (2008), no. 2, 545--562.

\end{thebibliography}
\end{document}